\documentclass[12pt,a4paper]{amsart}

\usepackage[T1]{fontenc}
\usepackage{lmodern}
\usepackage{a4wide,fullpage,setspace}
\usepackage{amsmath,amssymb,mathtools}
\usepackage[shortlabels]{enumitem}
\usepackage[numbers,sort&compress]{natbib}
\usepackage{tikz}
\usepackage{tikz-cd}
\usetikzlibrary{decorations.pathmorphing}
\definecolor{dark-red}{rgb}{0.5,0.15,0.15}
\usepackage[colorlinks=true,linkcolor=black,citecolor=dark-red,urlcolor=dark-red]{hyperref}
\hypersetup{pdftitle={Left properness of Moore flows},pdfauthor={Philippe Gaucher}}
\swapnumbers

\makeatletter
\let\P\@undefined
\let\leq\@undefined
\let\geq\@undefined
\let\vec\@undefined
\let\phi\@undefined
\let\epsilon\@undefined
\makeatother
\newcommand{\leq}{\leqslant}
\newcommand{\geq}{\geqslant}
\newcommand{\vec}{\overrightarrow}
\newcommand{\phi}{\varphi}
\newcommand{\epsilon}{\varepsilon}

\newcommand{\Top}{\mathop{\mathsf{Top}}}
\newcommand{\TOP}{\mathop{\mathsf{TOP}}}
\newcommand{\Pcat}{\mathcal P}
\newcommand{\Acat}{\mathcal A}
\newcommand{\G}{\mathcal G}
\newcommand{\M}{\mathcal M}
\newcommand{\One}{\mathbf 1}
\newcommand{\PSp}{[\Pcat^{\mathrm{op}},\Top]_0}
\newcommand{\PFlow}{\Pcat\mathop{\mathsf{Flow}}}
\newcommand{\Path}{\mathbb P}
\newcommand{\Ffree}{\mathbb F}
\newcommand{\id}{\operatorname{id}}
\newcommand{\Obj}{\operatorname{Obj}}
\newcommand{\Pmult}{\Pcat_{\times}}
\newcommand{\cocartesian}{\arrow[lu, phantom, "\ulcorner"{font=\Large}, pos=0]}

\newcommand{\liminj}{\varinjlim}

\newtheorem{thm}{Theorem}[section]
\newtheorem{lem}[thm]{Lemma}
\newtheorem{proposition}[thm]{Proposition}
\newtheorem{cor}[thm]{Corollary}
\theoremstyle{definition}
\newtheorem{definition}[thm]{Definition}

\title{Left properness of Moore flows}
\author[P. Gaucher]{Philippe Gaucher}
\address{Universit\'e Paris Cit\'e, CNRS, IRIF, F-75013, Paris, France}
\urladdr{\url{https://www.irif.fr/~gaucher}}
\begin{document}

\begin{abstract}
We introduce reparametrization categories with cuts. For every such category \(\Pcat\), the tensor product of two objectwise weak homotopy equivalences of \(\Pcat\)-spaces is an objectwise weak homotopy equivalence, and the q-model structure on \(\Pcat\)-flows is left proper. The reparametrization categories \(\G\) and \(\M\), as well as the terminal category \(\One\), have cuts. The terminal category recovers the left properness of the q-model structure on ordinary flows. We also give a symmetric reparametrization category with cuts and a strict unit. This category is not equivalent to \(\One\).
\end{abstract}

\maketitle
\setcounter{tocdepth}{1}
\tableofcontents
\hypersetup{linkcolor = dark-red}

\section{Introduction}

\subsection*{Presentation}
Directed algebraic topology studies topological spaces with a direction. A directed space is a topological space together with a set of continuous paths, containing all constant paths and stable under concatenation and nondecreasing reparametrizations. A directed path from \(x\) to \(y\) need not have a reverse path from \(y\) to \(x\). The fundamental category therefore replaces the fundamental groupoid of classical algebraic topology. We refer to \cite[Introduction and Part~I]{Grandis} for these definitions and for other models of non-reversible phenomena.

The study of concurrent systems is a motivation for directed algebraic topology. The points of a space represent states, and directed paths represent execution paths. Higher-dimensional cubes represent independent actions. Their directed paths describe different orders of execution. The homotopy type of a space of execution paths contains information which is lost by considering only the underlying topological space. We refer to \cite{GoubaultUserGuide,FajstrupRaussenGoubault,DAT_book} for examples and for the relation between directed algebraic topology and concurrency.

An ordinary flow \(X\) consists of a set \(X^0\) of states and, for each ordered pair \((\alpha,\beta)\), a topological space \(\mathbb P_{\alpha,\beta}X\) of execution paths from \(\alpha\) to \(\beta\), together with associative composition maps. No identity paths are specified: a flow is a small topologically enriched semicategory. The topology of \(\mathbb P_{\alpha,\beta}X\) describes continuous families of execution paths with fixed endpoints.

The concatenation of paths parametrized by \([0,1]\) is not strictly associative. Moore paths have a positive length, and the concatenation of paths of lengths \(a\) and \(b\) has length \(a+b\). This concatenation is strictly associative. A reparametrization category \(\Pcat\) specifies the lengths and the reparametrizations. A \(\Pcat\)-space is an enriched presheaf on \(\Pcat\), and a \(\Pcat\)-flow is a small semicategory enriched in the category of \(\Pcat\)-spaces with its tensor product. The objects of \(\G\) and \(\M\) are the positive real numbers, with addition as tensor product. The morphisms of \(\G\) are the increasing homeomorphisms of intervals; those of \(\M\) are the nondecreasing continuous surjections. We recall the construction and the q-model structure of \(\Pcat\)-flows from \cite[Sections~4--8]{Moore1}.

This paper extends \cite{leftproperflow} to reparametrization categories with cuts. The cartesian product of spaces is replaced by the tensor product of \(\Pcat\)-spaces. Since the latter is defined by a coend, we first prove that it preserves weak equivalences between arbitrary \(\Pcat\)-spaces, without a cofibrancy hypothesis. For the terminal category, the tensor product is the cartesian product, and our main theorem gives \cite[Theorem~5.6]{leftproperflow}.

Recall that a model category is \emph{left proper} if the pushout of a weak equivalence along a cofibration is a weak equivalence \cite[Digression~14.3.5]{Riehl}. In the following pushout square, if \(i\) is a cofibration and \(s\) is a weak equivalence, then \(\bar s\) is a weak equivalence:
\[
\begin{tikzcd}
 A \arrow[r,"s","\sim"'] \arrow[d,"i"'] & A' \arrow[d] \\
 X \arrow[r,"\bar s"',"\sim"] & A'\amalg_A X \cocartesian
\end{tikzcd}
\]
Thus attaching cells preserves weak equivalences. Homotopy colimits and left Bousfield localizations provide further applications; see \cite[Chapters~5--6]{Riehl} and \cite[Chapter~4]{ref_model2}.

Weak equivalences of Moore flows are detected objectwise on their path \(\Pcat\)-spaces. Colimits of Moore flows, however, can create execution paths. Attaching a globe adds its paths and their composites with the paths of the original flow. Identifying two states can also make new compositions possible. The proof therefore uses diagrams of finite tensor products and their latching maps, as in \cite{leftproperflow,Moore1}. The factors corresponding to the original flow need not be cofibrant. The tensor lemma of Section~\ref{sec:5} supplies the required preservation of weak equivalences, and Section~\ref{sec:6} proves left properness by a Reedy argument and transfinite induction.

\subsection*{Main results.}
Throughout, \(\Top\) is either the category of \(\Delta\)-generated spaces or the category of \(\Delta\)-Hausdorff \(\Delta\)-generated spaces. The main results are the following.
\begin{enumerate}
\item A \emph{cut structure} on a reparametrization category \(\Pcat\) provides continuous operators sending a map \(L\to a\otimes b\) to maps \(L\to a\) and \(L\to b\), together with a homotopy from the original map to the composite defined by its cuts. The operators and the homotopy are natural in \(a\) and \(b\). For arbitrary \(\Pcat\)-spaces \(D,E\) and every object \(L\), Theorem~\ref{thm:common-comparison} constructs maps
	\[
	r_L:(D\otimes E)(L)\longrightarrow D(L)\times E(L),
	\qquad
	s_L: D(L)\times E(L)\longrightarrow(D\otimes E)(L),
	\]
and proves that they are homotopy inverses. The maps and the homotopies are natural in \(D\) and \(E\).

\item Theorem~\ref{thm:tensor} states that if \(f: D\to D'\) and \(g: E\to E'\) are objectwise weak homotopy equivalences, then
	\[
	f\otimes g: D\otimes E\longrightarrow D'\otimes E'
	\]
is an objectwise weak homotopy equivalence.

\item Theorem~\ref{thm:left-proper} proves that, for every reparametrization category \(\Pcat\) with cuts, the q-model structure on \(\Pcat\)-flows is left proper.

\item Theorem~\ref{thm:GM-cuttable} proves that both reparametrization categories \(\G\) and \(\M\) have cuts. Corollary~\ref{cor:GM-left-proper} therefore gives the left properness of the q-model structures of \(\G\)-flows and \(\M\)-flows. Proposition~\ref{prop:one-cuttable} proves that the terminal category \(\mathbf 1\) has cuts, and Corollary~\ref{cor:one-left-proper} recovers the left properness of ordinary flows.

\item Proposition~\ref{prop:GM-nonunital} proves that no object \(e\) in \(\G\) or \(\M\) admits a natural isomorphism \(e\otimes L\cong L\). Theorem~\ref{thm:example} gives a symmetric reparametrization category with cuts and a strict unit: the category \(\Pcat_{\times}\) with one object and endomorphism space \([0,1]\). In this category, composition and tensor product are multiplication. Its underlying category is not equivalent to \(\One\).
\end{enumerate}

\subsection*{Outline of the paper.}
Section~\ref{sec:2} recalls the categories of spaces, the q- and h-model structures, and relative-\(T_1\) inclusions. Section~\ref{sec:3} recalls reparametrization categories, \(\Pcat\)-spaces, and Moore flows. Section~\ref{sec:4} defines cuts. Section~\ref{sec:5} proves the comparison theorem and the tensor lemma. Section~\ref{sec:6} proves left properness. Section~\ref{sec:7} constructs cuts for \(\G\) and \(\M\), and Section~\ref{sec:8} gives the case of ordinary flows. Finally, Section~\ref{sec:9} proves that \(\G\) and \(\M\) have no unit and constructs a symmetric reparametrization category with a strict unit. This category is not equivalent to the terminal category.

\section{Conventions on spaces and model categories}
\label{sec:2}

Throughout the paper, \(\Top\) denotes either the category of \(\Delta\)-generated spaces or its full subcategory of \(\Delta\)-Hausdorff \(\Delta\)-generated spaces.  A \(\Delta\)-generated space is a space whose topology is final with respect to all continuous maps from the standard topological simplices.  A \(\Delta\)-generated space is \(\Delta\)-Hausdorff if the image of every continuous map \([0,1]\to X\) is closed.  These definitions and both choices of \(\Top\) are studied in \cite[Section~2]{Moore1} and \cite[Section~2 and Appendix~B]{leftproperflow}.  In either case \(\Top\) is complete, cocomplete, locally presentable, and cartesian closed.  Its internal mapping space \(\TOP(-,-)\) is obtained by applying \(\Delta\)-kelleyfication to the compact-open topology \cite[Section~2, pp.~3--4]{Moore1}.

We use two model structures on \(\Top\).  In the q-model structure, the weak equivalences are the weak homotopy equivalences and the fibrations are the Serre fibrations.  Here a weak homotopy equivalence means a map inducing a bijection on path components and isomorphisms on all based homotopy groups in positive degrees. For either choice of \(\Top\), its canonical enrichment over itself is topologically bicomplete, with tensor \(X\otimes K=X\times K\) and cotensor \(X^{K}=\TOP(K,X)\); compare \cite[Definitions~1.4--1.5]{SchwanzlVogt}.  Local presentability implies the monomorphism hypothesis by \cite[Remark~5.20]{Barthel-Riel}, and hence the h-model structure exists by \cite[Corollary~5.23]{Barthel-Riel}. Its weak equivalences are the homotopy equivalences, and its cofibrations, called h-cofibrations, are the strong cofibrations of \cite[Definition~5.3]{Barthel-Riel}. Every q-cofibration is an h-cofibration; see, e.g., \cite[Remark~2.1]{leftproperflow}.

A one-to-one map \(i: A\to X\) is a \emph{relative-\(T_1\) inclusion} if, for every open \(U\subset A\) and every \(z\in X\setminus i(U)\), there is an open \(W\subset X\) such that \(i(U)\subset W\) and \(z\notin W\) \cite[p. 686]{hocolimfacile}.  Every h-cofibration is relative-\(T_1\) by \cite[Definition~2.2 and Proposition~2.6]{leftproperflow}.

A \textit{transfinite tower} of length \(\lambda\) is a functor \(X:\lambda\to\mathcal C\) such that, for every nonzero limit ordinal \(\mu<\lambda\), the canonical map
\[
\liminj_{\nu<\mu}X_\nu\longrightarrow X_\mu
\]
is an isomorphism \cite[Section~1, Notations]{leftproperflow}.

\section{Reparametrization categories and Moore flows}
\label{sec:3}

\begin{definition}[reparametrization category]\label{def:reparametrization}
A \emph{reparametrization category} \((\Pcat,\otimes)\) is a small \(\Top\)-enriched semimonoidal category satisfying the following axioms.
\begin{enumerate}
\item The semimonoidal structure is strict; in particular, \(\otimes\) induces a semigroup structure on \(\Obj(\Pcat)\).
\item Every enriched hom-space \(\Pcat(\ell,\ell')\) is contractible.
\item If \(\phi:\ell\to\ell'\) and \(\ell'=\ell'_1\otimes\ell'_2\), there are objects \(\ell_1,\ell_2\) and maps \(\phi_i:\ell_i\to\ell'_i\) such that \(\ell=\ell_1\otimes\ell_2\) and \(\phi=\phi_1\otimes\phi_2\).
\end{enumerate}
The tensor product on morphisms is the continuous map
\[
\otimes:\Pcat(\ell_1,\ell_2) \times \Pcat(\ell'_1,\ell'_2) \longrightarrow \Pcat(\ell_1\otimes \ell'_1,\ell_2\otimes \ell'_2).
\]
\end{definition}

This is the axiomatization of \cite[Definition~4.3 and Appendix~A]{Moore1}; it is also recalled in \cite[Definition~1]{Moore3}.  A \(\Pcat\)-space is an enriched functor \(D:\Pcat^{\mathrm{op}}\to\Top\).  Thus a map \(\alpha: a\to a'\) induces a contravariant action \(D(\alpha): D(a')\to D(a)\), and the adjoint action map \(\Pcat(a,a')\times D(a')\to D(a)\) is continuous.  The free \(\Pcat\)-space generated by \(U\in\Top\) at \(\ell\in\Obj(\Pcat)\) is
\begin{equation}\label{eq:free}
 \Ffree_\ell U=\Pcat(-,\ell)\times U;
\end{equation}
see \cite[Notation~5.15]{Moore1}.

The category \(\PSp\) has the projective q-model structure.  Its weak equivalences and fibrations are objectwise, and a set of generating q-cofibrations is
\begin{equation}\label{eq:Pspace-generators}
 \Ffree_\ell \mathsf{S}^{n-1}\longrightarrow \Ffree_\ell \mathsf{D}^n
 \qquad(\ell\in\Obj(\Pcat),\ n\geq0),
\end{equation}
where \(\mathsf{S}^{-1}=\varnothing\) and \(\mathsf{D}^0=\{0\}\); see \cite[Theorem~6.2]{dgrtop} and \cite[Section~8]{Moore1}.  This projective model structure is proper: Definition~\ref{def:reparametrization}(2) makes \(\Pcat\) locally contractible in the terminology of \cite{dgrtop}, and properness follows from \cite[Corollary~6.5]{dgrtop}.

For two \(\Pcat\)-spaces, their semimonoidal product is
\begin{equation}\label{eq:coend}
 (D\otimes E)(L) =\int^{a,b\in\Obj(\Pcat)}
   \Pcat(L,a\otimes b)\times D(a)\times E(b).
\end{equation}
This formula and the biclosed semimonoidal structure are established in \cite[Definition~5.7 and Theorem~5.14]{Moore1}.  Write \([\psi,x,y]\) for the coend class of \((\psi,x,y)\).  For maps \(\alpha: a\to a'\) and \(\beta: b\to b'\), the defining relation is
\begin{equation}\label{eq:relation}
 [\,(\alpha\otimes\beta)\psi,x',y'\,]
 =[\,\psi,D(\alpha)x',E(\beta)y'\,].
\end{equation}
This quotient description is \cite[Corollary~5.13]{Moore1} for \(n=2\).

\begin{definition}\label{def:Pflow}
A \textit{\(\Pcat\)-flow} is a small semicategory enriched in the biclosed semimonoidal category \((\PSp,\otimes)\).  Thus it consists of a set \(X^0\) of states, a \(\Pcat\)-space \(\Path_{\alpha,\beta}X\) for every ordered pair of states, and associative enriched composition maps
\[
 \Path_{\alpha,\beta}X\otimes\Path_{\beta,\gamma}X
 \longrightarrow\Path_{\alpha,\gamma}X.
\]
Its \(\Pcat\)-space of execution paths is \(\Path X=\coprod_{(\alpha,\beta)\in X^0\times X^0} \Path_{\alpha,\beta}X\). For a \(\Pcat\)-space \(D\), the \textit{globe} \(\operatorname{Glob}(D)\) has states \(0,1\), \(\Pcat\)-space \(D\) from \(0\) to \(1\), all other \(\Pcat\)-spaces of paths empty, and no nontrivial composition.
\end{definition}

These definitions are \cite[Definitions~6.1, 6.2 and~6.5]{Moore1}.  The q-model structure on \(\PFlow\) is cofibrantly generated.  A map \(f: X\to Y\) of \(\PFlow\) is a weak equivalence precisely when \(f^0\) is a bijection and every map
\[
 \Path_{\alpha,\beta}X\longrightarrow
 \Path_{f(\alpha),f(\beta)}Y
\]
is an objectwise weak homotopy equivalence.  Its generating q-cofibrations are
\begin{equation}\label{eq:Pflow-generators}
\begin{split}
 C&:\varnothing\longrightarrow\{0\},\\
 R&:\{0,1\}\longrightarrow\{0\},\\
 \operatorname{Glob}(\Ffree_\ell \mathsf{S}^{n-1})
 &\longrightarrow \operatorname{Glob}(\Ffree_\ell \mathsf{D}^n)
 \qquad(\ell\in\Obj(\Pcat),\ n\geq0).
\end{split}
\end{equation}
The \(\Pcat\)-flows in the first two lines have no paths.  These assertions are \cite[Notation~8.2, Theorems~8.8 and~8.16]{Moore1}.  The convention \(n=0\) includes \(\varnothing=\mathsf{S}^{-1}\to \mathsf{D}^0\).

\section{Reparametrization categories with cuts}
\label{sec:4}

\begin{definition}[reparametrization category with cuts]
\label{def:cuttable}
Let \(\Pcat\) be a reparametrization category. A \emph{cut structure} on \(\Pcat\) consists of the following data.
\begin{itemize}
\item For every object \(L\), a map
\[
 \delta_L: L\longrightarrow L\otimes L.
\]
\item For all \(L,a,b\in\Obj(\Pcat)\), continuous maps
\[
 c^{a,b}_1:\Pcat(L,a\otimes b)\longrightarrow\Pcat(L,a),
 \qquad
 c^{a,b}_2:\Pcat(L,a\otimes b)\longrightarrow\Pcat(L,b),
\]
called the \emph{cut operators}. They are natural in \(a\) and \(b\): for \(\alpha: a\to a'\), \(\beta: b\to b'\), and \(\psi: L\to a\otimes b\),
\begin{align}
 c^{a',b'}_1((\alpha\otimes\beta)\psi)
   &=\alpha c^{a,b}_1(\psi),                         \label{eq:cut-equiv1}\\
 c^{a',b'}_2((\alpha\otimes\beta)\psi)
   &=\beta c^{a,b}_2(\psi).                         \label{eq:cut-equiv2}
\end{align}
Equivalently, the following square commutes:
\[
\begin{tikzcd}[column sep=large,row sep=large]
 \Pcat(L,a\otimes b) \arrow[r,"{(c^{a,b}_1,c^{a,b}_2)}"] \arrow[d,"{(\alpha\otimes\beta)\circ-}"']
 & \Pcat(L,a)\times\Pcat(L,b) \arrow[d,"{(\alpha\circ-)\times(\beta\circ-)}"] \\
 \Pcat(L,a'\otimes b') \arrow[r,"{(c^{a',b'}_1,c^{a',b'}_2)}"']
 & \Pcat(L,a')\times\Pcat(L,b').
\end{tikzcd}
\]
\item Put
\begin{equation}\label{eq:normalization}
 N^{a,b}_L(\psi)=
 \bigl(c^{a,b}_1(\psi)\otimes c^{a,b}_2(\psi)\bigr)\delta_L.
\end{equation}
Thus the following triangle commutes:
\[
\begin{tikzcd}[column sep=large]
 L \arrow[r,"\delta_L"] \arrow[dr,"N^{a,b}_L(\psi)"']
   & L\otimes L \arrow[d,"c^{a,b}_1(\psi)\otimes c^{a,b}_2(\psi)"] \\
   & a\otimes b
\end{tikzcd}
\]
There is a continuous homotopy
\[
 \Theta^{a,b}_L:
 [0,1]\times\Pcat(L,a\otimes b)\longrightarrow\Pcat(L,a\otimes b)
\]
from the identity to \(N^{a,b}_L\), natural in \(a\) and \(b\):
\begin{align}
 &\Theta^{a,b}_L(0,\psi)=\psi,\qquad
 \Theta^{a,b}_L(1,\psi)=N^{a,b}_L(\psi),
                                                        \label{eq:Theta-ends}\\
 &\Theta^{a',b'}_L(u,(\alpha\otimes\beta)\psi)
   =(\alpha\otimes\beta)\Theta^{a,b}_L(u,\psi).
                                                        \label{eq:Theta-equiv}
\end{align}
\end{itemize}
A \textit{reparametrization category with cuts} is a reparametrization category which admits a cut structure. We also say that a reparametrization category \textit{has cuts}.
\end{definition}

Sections~\ref{sec:7}--\ref{sec:9} give examples of reparametrization categories with cuts. The following example shows that a reparametrization category need not have cuts.

\begin{lem}\label{lem:one-object-reparametrization}
Let \(B\) be a contractible commutative topological monoid in \(\Top\), with identity \(1\). The one-object enriched category with hom-space \(B\) and composition given by multiplication is a symmetric reparametrization category with a strict unit when its tensor product is multiplication.
\end{lem}

\begin{proof}
The interchange identity
\[
(b_2b_1)\otimes(c_2c_1)=b_2b_1c_2c_1
=b_2c_2b_1c_1=(b_2\otimes c_2)(b_1\otimes c_1)
\]
shows that multiplication defines an enriched bifunctor. Associativity and the identity laws give strict associativity and a strict unit. Take the identity morphism as symmetry. Its naturality is commutativity of multiplication, and the coherence identities reduce to identity laws; compare \cite[Definitions~8.1.1 and~8.1.12]{Tensor_categories}. The hom-space is contractible by hypothesis, and \(b=b\otimes1\) gives the factorization axiom.
\end{proof}

\begin{proposition}\label{prop:Acat}
Let \(\Acat\) have one object, hom-space \(\mathbb R\), composition and tensor both addition, and identity \(0\). The category \(\Acat\) is a symmetric reparametrization category with a strict unit, and it has no cuts.
\end{proposition}

\begin{proof}
The real line is a locally path-connected Hausdorff metrizable, hence \(\Delta\)-generated, space \cite[Proposition~3.11]{MR3270173}. It is a contractible commutative topological monoid under addition, so Lemma~\ref{lem:one-object-reparametrization} gives the structural assertions. Suppose a first cut operator \(c_1:\mathbb R\to\mathbb R\) existed. Its naturality would give
	\begin{equation}\label{eq:additive-cut}
	c_1(\alpha+\beta+\psi)=\alpha+c_1(\psi).
	\end{equation}
Taking \(\alpha=\psi=0\), and then \(\beta=\psi=0\), gives
\[
c_1(\beta)=c_1(0),\qquad c_1(\alpha)=\alpha+c_1(0).
\]
For every \(t\in\mathbb R\), these two identities imply
\[
c_1(0)=c_1(t)=t+c_1(0),
\]
which would force every \(t\in\mathbb R\) to vanish.
\end{proof}

\section{The tensor lemma}
\label{sec:5}

Fix a reparametrization category \(\Pcat\) with cuts, two arbitrary \(\Pcat\)-spaces \(D,E\), and \(L\in\Obj(\Pcat)\).  Define
\begin{equation}\label{eq:common-r}
 r_L:(D\otimes E)(L)\longrightarrow D(L)\times E(L)
\end{equation}
by
\begin{equation}\label{eq:common-r-formula}
 r_L[\psi,x,y]=
 \bigl(D(c^{a,b}_1(\psi))x,E(c^{a,b}_2(\psi))y\bigr).
\end{equation}
Using the \((a,b)=(L,L)\) coend summand, define
\begin{equation}\label{eq:common-s}
 s_L: D(L)\times E(L)\longrightarrow(D\otimes E)(L),\qquad
 s_L(x,y)=[\delta_L,x,y].
\end{equation}

\begin{proposition}\label{lem:comparison-maps}
The formulas \eqref{eq:common-r-formula} and \eqref{eq:common-s} define continuous maps, natural in \(D\) and \(E\).
\end{proposition}

\begin{proof}
Equations \eqref{eq:cut-equiv1}--\eqref{eq:cut-equiv2} give
\begin{align*}
r_L[(\alpha\otimes\beta)\psi,x',y']
&=\bigl(D(\alpha c^{a,b}_1(\psi))x',E(\beta c^{a,b}_2(\psi))y'\bigr)\\
&=\bigl(D(c^{a,b}_1(\psi))D(\alpha)x',E(c^{a,b}_2(\psi))E(\beta)y'\bigr)\\
&=r_L[\psi,D(\alpha)x',E(\beta)y'].
\end{align*}
Thus \(r_L\) descends through \eqref{eq:relation}; it is continuous by the universal property of the coend and continuity of the cut operators and enriched actions. The map \(s_L\) is the composite of the inclusion at \(\delta_L\) and the canonical map from the \((L,L)\) coend summand. Enriched naturality of \(f:D\to D'\) and \(g:E\to E'\) gives
\[
(f_L\times g_L)r_L=r'_L(f\otimes g)_L,
\qquad
(f\otimes g)_Ls_L=s'_L(f_L\times g_L).\qedhere
\]
\end{proof}

\begin{thm}[comparison theorem]\label{thm:common-comparison}
For every reparametrization category with cuts \(\Pcat\), arbitrary \(\Pcat\)-spaces \(D,E\), and every \(L\in\Obj(\Pcat)\), the maps \(r_L\) and \(s_L\) are homotopy inverses.  Thus, naturally in \(D\) and \(E\),
\[
\begin{tikzcd}[column sep=huge]
 (D\otimes E)(L) \arrow[r,shift left=1.1ex,"r_L"]
   & D(L)\times E(L) \arrow[l,shift left=1.1ex,"s_L"]
\end{tikzcd}
\qquad r_Ls_L\simeq\id,\quad s_Lr_L\simeq\id. \]
\end{thm}

\begin{proof}
Put \(\lambda_i=c^{L,L}_i(\delta_L)\). Contractibility of \(\Pcat(L,L)\) gives a homotopy \(h:[0,1]\times\Pcat(L,L)\to\Pcat(L,L)\) from the identity to the constant map at \(\id_L\). The enriched actions define
\[
H(u,x,y)=\bigl(D(h(u,\lambda_1))x,E(h(u,\lambda_2))y\bigr),
\qquad H_0=r_Ls_L,\quad H_1=\id.
\]
For \(c_i=c_i^{a,b}(\psi)\), relation \eqref{eq:relation} gives
\begin{equation}\label{eq:common-sr}
\begin{aligned}
s_Lr_L[\psi,x,y]
&=[\delta_L,D(c_1)x,E(c_2)y]\\
&=[(c_1\otimes c_2)\delta_L,x,y]
=[N^{a,b}_L(\psi),x,y].
\end{aligned}
\end{equation}
By \eqref{eq:Theta-equiv},
\begin{align*}
[\Theta^{a',b'}_L(u,(\alpha\otimes\beta)\psi),x',y']
&=[(\alpha\otimes\beta)\Theta^{a,b}_L(u,\psi),x',y']\\
&=[\Theta^{a,b}_L(u,\psi),D(\alpha)x',E(\beta)y'].
\end{align*}
Hence \(\Theta\) descends to
\[
K(u,[\psi,x,y])=[\Theta^{a,b}_L(u,\psi),x,y],
\qquad K_0=\id,\quad K_1=s_Lr_L.
\]
This map is continuous since \([0,1]\times-\) preserves the colimits defining the coend. Both homotopies are natural in \(D,E\) by enriched naturality.
\end{proof}

\begin{thm}[tensor lemma]\label{thm:tensor}
Let \(\Pcat\) be a reparametrization category with cuts.  If \(f: D\to D'\) and \(g: E\to E'\) are objectwise weak homotopy equivalences of arbitrary \(\Pcat\)-spaces, then
\[
 f\otimes g: D\otimes E\longrightarrow D'\otimes E'
\]
is an objectwise weak homotopy equivalence.
\end{thm}

\begin{proof}
For every \(L\in\Obj(\Pcat)\), naturality of \(r_L\) gives a commutative square
\[
\begin{tikzcd}
 (D\otimes E)(L) \arrow[r,"(f\otimes g)_L"] \arrow[d,"r_L"']
   & (D'\otimes E')(L) \arrow[d,"r'_L"] \\
 D(L)\times E(L) \arrow[r,"f_L\times g_L"']
   & D'(L)\times E'(L)
\end{tikzcd}
\]
The vertical maps are homotopy equivalences by Theorem~\ref{thm:common-comparison}, hence weak homotopy equivalences. The bottom map is a weak homotopy equivalence: path components of a finite product are products of path components, and maps and homotopies into a product give natural isomorphisms \cite[Proposition~4.2]{MR1867354}
\[
 \pi_n(X\times Y,(x,y))\cong
 \pi_n(X,x)\times\pi_n(Y,y)\qquad(n\geq1).
\]
The two-out-of-three property, applied to the square, proves that \((f\otimes g)_L\) is a weak homotopy equivalence.  Since \(L\) was arbitrary, the result is objectwise.
\end{proof}

\section{Left properness}
\label{sec:6}

For maps \(f: A\to B\) and \(g: C\to D\) of \(\Pcat\)-spaces, their pushout product is
\[
 f\mathbin{\square_\otimes}g:
 (A\otimes D)\mathop{\coprod}_{A\otimes C}(B\otimes C)
 \longrightarrow B\otimes D.
\]
The notation \(\square_\times\) will mean the analogous pushout product in \((\Top,\times)\). Iterated pushout products are associated from the left; by associativity, the other choices are canonically isomorphic to this one. This is \cite[Definition~9.1]{Moore1}.

\begin{proposition}[cubical formula for the iterated pushout product]\label{prop:calculpushout}
Let \(f_i:A_i\longrightarrow B_i\), for \(0\leq i\leq p\), be \(p+1\) maps of \(\Pcat\)-spaces. For \(S\subset\{0,\dots,p\}\), put
\[
 C_p(S)=V_0(S)\otimes\cdots\otimes V_p(S),
 \qquad V_i(S)=\begin{cases}
 A_i,&i\notin S,\\
 B_i,&i\in S.
 \end{cases}
\]
For \(S\subset T\), let \(C_p(i_S^T):C_p(S)\longrightarrow C_p(T)\) be the morphism
\[
 g_0\otimes\cdots\otimes g_p,
 \qquad g_i=\begin{cases}
 \id_{B_i},&i\in S,\\
 f_i,&i\in T\setminus S,\\
 \id_{A_i},&i\notin T.
 \end{cases}
\]
These assignments define a functor \(C_p:(2^{\{0,\ldots,p\}},\subset)\to\PSp\), and its canonical map
\[
\liminj_{S\subsetneq\{0,\ldots,p\}}C_p(S)
\longrightarrow C_p(\{0,\ldots,p\})
\]
is \(f_0\square_\otimes\cdots\square_\otimes f_p\).
\end{proposition}

\begin{proof}
This is \cite[Proposition~9.8]{Moore1}. Write \(Q_p=\liminj_{S\subsetneq\{0,\ldots,p\}}C_p(S)\). Functoriality follows from
\[
C_p(i_T^U)C_p(i_S^T)=C_p(i_S^U),
\qquad C_p(i_S^S)=\id_{C_p(S)}.
\]
For \(p\geq1\), put \(B_{<p}=B_0\otimes\cdots\otimes B_{p-1}\). Splitting the indexing poset according to whether \(p\in S\), and using separate preservation of colimits, gives the pushout
\[
\begin{tikzcd}[column sep=large]
Q_{p-1}\otimes A_p \arrow[r] \arrow[d,"\id\otimes f_p"']
& B_{<p}\otimes A_p \arrow[d]\\
Q_{p-1}\otimes B_p \arrow[r] & Q_p \cocartesian.
\end{tikzcd}
\]
Consequently,
\[
\bigl(Q_p\longrightarrow B_{<p}\otimes B_p\bigr)
=\bigl(Q_{p-1}\longrightarrow B_{<p}\bigr)\square_\otimes f_p,
\qquad \bigl(Q_0\longrightarrow B_0\bigr)=f_0.
\]
Induction proves the assertion.
\end{proof}

For \(p=1\), the construction is the following commutative diagram, whose left square is a pushout. Its lower right arrow is the pushout product:
\[
\begin{tikzcd}[column sep=large,row sep=large]
 A_0\otimes A_1 \arrow[r,"f_0\otimes\id"] \arrow[d,"\id\otimes f_1"']
   & B_0\otimes A_1 \arrow[d] \arrow[r,"\id\otimes f_1"]
   & B_0\otimes B_1 \arrow[d,equal] \\
 A_0\otimes B_1 \arrow[r]
   & Q \cocartesian \arrow[r,"f_0\square_\otimes f_1"']
   & B_0\otimes B_1
\end{tikzcd}
\]
Here \(Q=(A_0\otimes B_1)\amalg_{A_0\otimes A_1}(B_0\otimes A_1)\). Give the poset \((2^{\{0,\ldots,p\}},\subset)\) the direct Reedy structure with degree \(d(S)=|S|\). Proposition~\ref{prop:calculpushout} then identifies
\[
\bigl(L_{\{0,\ldots,p\}}C_p\longrightarrow C_p(\{0,\ldots,p\})\bigr)
=f_0\square_\otimes\cdots\square_\otimes f_p.
\]

\begin{lem}\label{lem:finite-tensor-coproduct}
Nonempty finite tensor products of objectwise weak homotopy equivalences of \(\Pcat\)-spaces are objectwise weak homotopy equivalences. For a family \((f_i:D_i\to E_i)_{i\in I}\), the map \(\coprod_{i\in I}f_i\) is an objectwise weak homotopy equivalence if and only if every \(f_i\) is one.
\end{lem}

\begin{proof}
Induct on the number of tensor factors using Theorem~\ref{thm:tensor}. For coproducts, the natural isomorphisms
\[
\begin{aligned}
\pi_0\!\left(\coprod_{i\in I}D_i(L)\right)&\cong\coprod_{i\in I}\pi_0(D_i(L)),\\
\pi_n\!\left(\coprod_{i\in I}D_i(L),x\right)&\cong\pi_n(D_i(L),x)
\qquad(x\in D_i(L),\ n\geq1)
\end{aligned}
\]
preserve the index \(i\) and prove both implications.
\end{proof}

\begin{proposition}[well known]\label{lem:h-cartesian-monoidal}
Every space is h-cofibrant, and the h-model structure on \(\Top\) is cartesian monoidal.
\end{proposition}

\begin{proof}
The interval is a good cylinder object since \(\{0,1\}\hookrightarrow[0,1]\) is a q-cofibration and hence an h-cofibration; see \cite[Example~1.11(A)(1) and Sections~4.1 and~4.3]{SchwanzlVogt}. The h-cofibrations are precisely the strong cofibrations of \cite{SchwanzlVogt} by \cite[Definition~5.3 and Corollary~5.23]{Barthel-Riel} and \cite[Proposition~3.5(6)]{SchwanzlVogt}. Applied to the cartesian product, \cite[Pairing Theorem~2.7(1) and Addendum~3.6(1)]{SchwanzlVogt} gives, for h-cofibrations \(i,j\),
\[
\begin{aligned}
i,j\text{ h-cofibrations}
&\Longrightarrow i\square_\times j\text{ an h-cofibration},\\
i\text{ or }j\text{ a homotopy equivalence}
&\Longrightarrow i\square_\times j\text{ a homotopy equivalence}.
\end{aligned}
\]
Every space is strongly cofibrant by \cite[Lemma~2.6(3)]{SchwanzlVogt}. In particular, the unit \(\{*\}\) is h-cofibrant, so the pushout-product and unit axioms hold \cite[Definition~4.2.6]{MR99h:55031}.
\end{proof}

\begin{lem}[well known]\label{lem:product-hcof}
If \(i: A\to B\) is an h-cofibration and \(K\) is any space, then \(K\times i: K\times A\to K\times B\) is an h-cofibration.  Pushouts and transfinite composites of h-cofibrations are h-cofibrations.
\end{lem}

\begin{proof}
The map \(K\times i\) is the pushout product \((\varnothing\to K)\square_\times i\), so it is an h-cofibration by Proposition~\ref{lem:h-cartesian-monoidal}. Pushouts and transfinite composites of cofibrations are cofibrations in every model category \cite[Lemma~2.1.10]{MR99h:55031}.
\end{proof}

\begin{proposition}[Tensor products and copowers]\label{lem:tensor-copowers}
For \(\Pcat\)-spaces \(D,E\) and a space \(U\), the pointwise copower \(D\times U\) satisfies natural homeomorphisms
	\begin{align}
		((D\times U)\otimes E)(L)&\cong(D\otimes E)(L)\times U,\label{eq:copower-left}\\
		(D\otimes(E\times U))(L)&\cong(D\otimes E)(L)\times U.\label{eq:copower-right}
	\end{align}
\end{proposition}

\begin{proof}
The functor \(-\times U\) preserves colimits, so
\begin{align*}
((D\times U)\otimes E)(L)
&=\int^{a,b}\Pcat(L,a\otimes b)\times(D(a)\times U)\times E(b)\\
&\cong\left(\int^{a,b}\Pcat(L,a\otimes b)\times D(a)\times E(b)\right)\times U\\
&=(D\otimes E)(L)\times U.
\end{align*}
The same calculation with \(U\) in the second factor proves \eqref{eq:copower-right}. Equation \eqref{eq:free} gives
\begin{equation}\label{eq:free-copower}
\Ffree_\ell U=(\Ffree_\ell\{*\})\times U.\qedhere
\end{equation}
\end{proof}

\begin{samepage}
\begin{proposition}[Pushout products and pushouts]\label{lem:pushout-product-cobase}
Suppose that a tensor product preserves colimits in each variable. Replacing either map in a pushout product by a pushout replaces the pushout product by a pushout. The same holds for any factor of an iterated pushout product.
\end{proposition}
\end{samepage}

\begin{proof}
Let \(f:A\to B\), \(g:C\to D\), and \(f':A'\to B'=B\amalg_A A'\), and write
\[
Q_f=(A\otimes D)\amalg_{A\otimes C}(B\otimes C),\qquad
Q_{f'}=(A'\otimes D)\amalg_{A'\otimes C}(B'\otimes C).
\]
Separate preservation of colimits gives
\begin{align*}
Q_{f'}
&\cong(A'\otimes D)\amalg_{A'\otimes C}
\bigl((B\otimes C)\amalg_{A\otimes C}(A'\otimes C)\bigr)\\
&\cong Q_f\amalg_{A\otimes D}(A'\otimes D),\\
(B\otimes D)\amalg_{Q_f}Q_{f'}
&\cong(B\otimes D)\amalg_{A\otimes D}(A'\otimes D)
\cong B'\otimes D.
\end{align*}
Consequently, the following square is a pushout:
\[
\begin{tikzcd}[column sep=large]
Q_f \arrow[r,"f\square_\otimes g"] \arrow[d]
& B\otimes D \arrow[d]\\
Q_{f'} \arrow[r,"f'\square_\otimes g"'] & B'\otimes D \cocartesian.
\end{tikzcd}
\]
The same calculation applies to the second variable. Induction on the number of factors proves the last assertion.
\end{proof}

The following proposition applies these calculations to the latching maps.

\begin{proposition}[latching maps for free spaces]\label{lem:free-latching}
Let \(p,k\geq0\) with \(p+k>0\), and let \(W_1,\dots,W_p\) be arbitrary \(\Pcat\)-spaces.  For \(1\leq r\leq k\), choose \(\ell_r\in\Obj(\Pcat)\) and \(n_r\geq0\), and let
	\[
	j_r=\Ffree_{\ell_r}
	(\mathsf{S}^{n_r-1}\hookrightarrow \mathsf{D}^{n_r})
	\]
be a generating projective q-cofibration of \(\Pcat\)-spaces.  Every evaluation at \(L\in\Obj(\Pcat)\) of an iterated pushout product made, in any fixed order, from the maps \(\varnothing\to W_i\) and the maps \(j_r\) is an h-cofibration.  The same conclusion holds if any of the maps \(j_r\) is replaced by a pushout of \(j_r\) in \(\PSp\).
\end{proposition}

\begin{proof}
Number the maps in their fixed order by \(I=\{1,\ldots,p+k\}\). Let \(I_W\) be the positions of \(\varnothing\to W_i\), let \(I_F=I\setminus I_W\), and let \(C\) be their cube from Proposition~\ref{prop:calculpushout}. Since \(\otimes\) preserves colimits in each variable,
\[
I_W\not\subset S\Longrightarrow C(S)=\varnothing.
\]
If \(k=0\), the iterated pushout product evaluates to \(\varnothing\to C(I)(L)\), an h-cofibration by Proposition~\ref{lem:h-cartesian-monoidal}. Suppose \(k>0\). Let \(K\) be the tensor product obtained from \(C(I)\) by replacing each \(\Ffree_{\ell_r}\mathsf D^{n_r}\) by \(\Ffree_{\ell_r}\{*\}\). For \(T\subset I_F\), put
\[
U_r(T)=
\begin{cases}
\mathsf D^{n_r},&\text{the position of \(j_r\) belongs to \(T\)},\\
\mathsf S^{n_r-1},&\text{otherwise}.
\end{cases}
\]
Equations \eqref{eq:copower-left}--\eqref{eq:free-copower} identify the entire diagram naturally in \(T\):
\begin{equation}\label{eq:free-cube-vertices}
C(I_W\cup T)(L)\cong K(L)\times\prod_{r=1}^kU_r(T).
\end{equation}
Set
\[
B=\prod_{r=1}^k\mathsf D^{n_r},
\qquad
\partial B=\bigcup_{r=1}^k
\mathsf D^{n_1}\times\cdots\times\mathsf S^{n_r-1}\times\cdots\times\mathsf D^{n_k}.
\]
Evaluation and \(K(L)\times-\) preserve colimits \cite[Proposition~5.3]{dgrtop}. Empty vertices contribute nothing, so \eqref{eq:free-cube-vertices} gives
\begin{equation}\label{eq:free-latching-product}
\begin{tikzcd}[column sep=huge]
 \bigl(\liminj_{S\subsetneq I}C(S)\bigr)(L) \arrow[r] \arrow[d,"\cong"']
   & C(I)(L) \arrow[d,"\cong"] \\
 K(L)\times\partial B \arrow[r,hook] & K(L)\times B.
\end{tikzcd}
\end{equation}
The inclusion \(\partial B\hookrightarrow B\) is the iterated cartesian pushout product of the inclusions \(\mathsf S^{n_r-1}\hookrightarrow\mathsf D^{n_r}\). It is a CW subcomplex inclusion, hence a q-cofibration in either choice of \(\Top\); compare \cite[Proposition~4.2.11]{MR99h:55031}. It is therefore an h-cofibration, and Lemma~\ref{lem:product-hcof} applies to the bottom map of \eqref{eq:free-latching-product}.

Replacing any \(j_r\) by a pushout gives a pushout of the iterated pushout product by Proposition~\ref{lem:pushout-product-cobase}. Evaluation preserves pushouts, and h-cofibrations are preserved by pushout.
\end{proof}

\begin{lem}[Reedy comparison]\label{lem:Reedy-comparison}
Let \(\mathcal R\) be a small Reedy category for which the colimit functor is left Quillen.  Let \(D,D':\mathcal R\to\Top\) be Reedy cofibrant for the h-model structure.  If \(D\to D'\) is objectwise a weak homotopy equivalence, then
\[
 \liminj_{\mathcal R}D\longrightarrow\liminj_{\mathcal R}D'
\]
is a weak homotopy equivalence.
\end{lem}

\begin{proof}
Ordinary colimits of Reedy h-cofibrant diagrams compute their h-model homotopy colimits. The q-model and h-model homotopy colimits have the same weak homotopy type by \cite[Theorem~A.7]{hocolimfacile}, adapted to \(\Top\) in \cite[Theorem~2.8]{leftproperflow}. Using functorial replacements gives the comparison square
\[
\begin{tikzcd}[column sep=huge]
 \operatorname{hocolim}_{q}D \arrow[r,"\sim"] \arrow[d,"\sim"']
   & \liminj D \arrow[d] \\
 \operatorname{hocolim}_{q}D' \arrow[r,"\sim"'] & \liminj D'.
\end{tikzcd}
\]
The left map is a weak homotopy equivalence by homotopy invariance of the derived colimit. The two-out-of-three property proves the assertion; this is the argument of \cite[proof of Theorem~5.3]{leftproperflow}.
\end{proof}

\begin{cor}[Transfinite compositions]\label{cor:transfinite-gluing}
Let \(X\to Y\) be a map of transfinite towers \(\lambda\to\Top\).  Suppose that all successor maps in both towers are h-cofibrations and that \(X_\mu\to Y_\mu\) is a weak homotopy equivalence for every \(\mu<\lambda\). Then \(\liminj X\to\liminj Y\) is a weak homotopy equivalence.
\end{cor}

\begin{proof}
The spaces \(\{*\}\), \([0,1]\), \(\mathsf S^n\), and \(\mathsf S^n\times[0,1]\), for \(n\geq1\), are finite relative to h-cofibrations \cite[Proposition~2.5 and Corollary~2.7]{leftproperflow}. The transition maps and maps to the colimits are one-to-one \cite[Section~2 and Proposition~B.16]{leftproperflow}. Factoring points, paths, based maps, and based homotopies through stages therefore gives
\[
\pi_0(\liminj X)\cong\liminj_{\mu<\lambda}\pi_0(X_\mu),
\qquad
\pi_0(\liminj Y)\cong\liminj_{\mu<\lambda}\pi_0(Y_\mu).
\]
For \(x\in X_{\mu_0}\), let \(x_\mu\) be its image at stage \(\mu\), and let \(s_\mu:X_\mu\to Y_\mu\) and \(s_\infty:\liminj X\to\liminj Y\) be the comparison maps. The same factorization gives, for \(n\geq1\),
\[
\begin{tikzcd}[column sep=huge]
 \liminj_{\mu_0\leq\mu<\lambda}\pi_n(X_\mu,x_\mu)
 \arrow[r,"\cong"] \arrow[d,"\cong"']
 & \pi_n(\liminj X,x) \arrow[d] \\
 \liminj_{\mu_0\leq\mu<\lambda}\pi_n(Y_\mu,s_\mu(x_\mu))
 \arrow[r,"\cong"']
 & \pi_n(\liminj Y,s_\infty(x)).
\end{tikzcd}
\]
The left map and the map on \(\pi_0\) are isomorphisms by hypothesis.
\end{proof}

We recall the Reedy description needed below.  Given a set \(S\) and \(u,v\in S\), let \(\mathcal R_{u,v}(S)\) denote the category written \(\Pcat_{u,v}(S)\) in \cite[Section~9]{Moore1}; the change of letter avoids confusion with the reparametrization category \(\Pcat\).  Its objects are finite words
\[
 ((u_0,\epsilon_1,u_1),\ldots,
   (u_{m-1},\epsilon_m,u_m)),
\]
where \(u_i\in S\), \(\epsilon_i\in\{0,1\}\), and \(\epsilon_i=1\) forces \((u_{i-1},u_i)=(u,v)\).  Its arrows are generated by contractions of two adjacent \(0\)-labelled objects (the composition maps) and by replacements \((u,0,v)\to(u,1,v)\) (the inclusion maps), subject to the associativity and interchange relations.  The complete presentation is \cite[Section~9, before Notation~9.5]{Moore1}.  It is a Reedy category with degree \(m+\sum_i\epsilon_i\); inclusion maps raise degree and composition maps lower it.  For a diagram \(D\), the latching object at \(\mathbf n\) is
\[
 L_{\mathbf n}D=
 \liminj_{\partial(\mathcal R_{u,v}(S)^+\downarrow\mathbf n)}D.
\]
The Reedy model structure exists for every model target, and its colimit functor is left Quillen \cite[Theorem~9.6]{Moore1}.

Consider a pushout of \(\Pcat\)-flows
\begin{equation}\label{eq:globe-pushout}
\begin{tikzcd}
 \operatorname{Glob}(\partial Z) \arrow[r,"g"] \arrow[d]
   & A \arrow[d,"f"] \\
 \operatorname{Glob}(Z) \arrow[r] & X \cocartesian
\end{tikzcd}
\end{equation}
and put \(u=g(0)\), \(v=g(1)\).  Define
\begin{equation}\label{eq:T-pushout}
\begin{tikzcd}[column sep=large]
 \partial Z \arrow[r] \arrow[d] & \Path_{u,v}A \arrow[d] \\
 Z \arrow[r] & T=Z\amalg_{\partial Z}\Path_{u,v}A \cocartesian
\end{tikzcd}
\end{equation}
The diagram \(D_f:\mathcal R_{u,v}(A^0)\to\PSp\) assigns to the word \(\mathbf n\) above the ordered tensor product
\begin{equation}\label{eq:Df}
 D_f(\mathbf n)=Z_1\otimes\cdots\otimes Z_m,
 \qquad
 Z_i=\begin{cases}
  \Path_{u_{i-1},u_i}A,&\epsilon_i=0,\\
  T,&\epsilon_i=1.
 \end{cases}
\end{equation}
The inclusion arrows use \(\Path_{u,v}A\to T\) and the contraction arrows use composition in \(A\).  This is well defined and
\begin{equation}\label{eq:colim-Df}
 \liminj D_f\cong\Path X
\end{equation}
by \cite[Theorem~9.7]{Moore1}.  Its latching map is the iterated pushout product of the maps \(\varnothing\to\Path_{a,b}A\) belonging to the factors with label \(0\) and the maps \(\Path_{u,v}A\to T\) belonging to the factors with label \(1\); if all labels are \(0\), the latching object is empty.  This is \cite[Proposition~9.9]{Moore1}. Colimits of enriched presheaves are computed objectwise \cite[Proposition~5.3]{dgrtop}; consequently, evaluation at \(L\) commutes with the latching colimits and with the colimit in \eqref{eq:colim-Df}.

\begin{proposition}[Attaching a globe]\label{prop:globe-step}
Let \(s: A\to A'\) be a weak equivalence of \(\Pcat\)-flows.  Let \(\partial Z\to Z\) be one of the generating maps \(\Ffree_\ell \mathsf{S}^{n-1}\to\Ffree_\ell \mathsf{D}^n\), and form the two pushouts
\[
\begin{tikzcd}[column sep=large,row sep=large]
 \operatorname{Glob}(\partial Z) \arrow[r] \arrow[d]
   & A \arrow[r,"s","\sim"'] \arrow[d,"f"'] & A' \arrow[d,"f'"] \\
 \operatorname{Glob}(Z) \arrow[r]
   & X \cocartesian \arrow[r,"\bar s"'] & X' \cocartesian
\end{tikzcd}
\]
Then the induced map \(\bar s: X\to X'\) is a weak equivalence.  Moreover \(\Path A(L)\to\Path X(L)\) is an h-cofibration for every \(L\in\Obj(\Pcat)\).
\end{proposition}

\begin{proof}
The bijection \(s^0\) identifies the state sets and attaching endpoints. Form the pushouts
\[
\begin{tikzcd}[column sep=large]
 \partial Z \arrow[r] \arrow[d]
   & \Path_{u,v}A \arrow[r,"\sim"] \arrow[d] & \Path_{u,v}A' \arrow[d] \\
 Z \arrow[r] & T \cocartesian \arrow[r,"\sim"'] & T' \cocartesian.
\end{tikzcd}
\]
The map \(T\to T'\) is an objectwise weak homotopy equivalence by left properness of the projective q-model structure \cite[Corollary~6.5]{dgrtop}.

The diagrams \(D_f,D_{f'}\) have the same Reedy indexing category. Their latching maps are iterated pushout products of maps \(\varnothing\to\Path_{a,b}A\) and \(\Path_{u,v}A\to T\), or the corresponding maps for \(A'\). The latter maps are pushouts of \(\partial Z\to Z\), so Proposition~\ref{lem:free-latching} makes both evaluated diagrams Reedy h-cofibrant. Every component of \(D_f\to D_{f'}\) is a finite tensor product of objectwise weak homotopy equivalences, hence is one by Lemma~\ref{lem:finite-tensor-coproduct}. Lemma~\ref{lem:Reedy-comparison} and \eqref{eq:colim-Df} give
\[
\Path X(L)\cong\liminj D_f(L)
\xrightarrow{\sim}\liminj D_{f'}(L)\cong\Path X'(L).
\]
Since \(X^0=A^0\) and \(X'^0=A'^0\), the map \(\bar s\) is bijective on states. Its path map preserves the summands indexed by pairs of states, so Lemma~\ref{lem:finite-tensor-coproduct} proves that \(\bar s\) is a weak equivalence.

For the cofibration assertion, use \(D_{\id_A}\to D_f\) from \cite[proof of Theorem~9.10]{Moore1}. Its relative latching map is an isomorphism for words with only zero labels; otherwise it is an iterated pushout product of maps \(\varnothing\to\Path_{a,b}A\) and at least one copy of \(\Path_{u,v}A\to T\) \cite[proof of Theorem~9.10, cases (a) and (b)]{Moore1}. Proposition~\ref{lem:free-latching} makes its evaluation a Reedy h-cofibration. The left Quillen colimit functor \cite[Theorem~9.6]{Moore1} therefore gives the h-cofibration
\[
\Path A(L)\cong\liminj D_{\id_A}(L)
\longrightarrow\liminj D_f(L)\cong\Path X(L).\qedhere
\]
\end{proof}

We first separate the description of the pushout along \(R\) from its homotopical consequences. This extends \cite[proof of Theorem~5.6, case~(2)]{leftproperflow}.

Fix distinct states \(a,b\in A^0\), let \(q:A^0\to A^0/{\sim}\) identify \(a,b\), and form the pushout
\[
\begin{tikzcd}
 \{0,1\} \arrow[r,"{0\mapsto a,\,1\mapsto b}"] \arrow[d,"R"'] & A \arrow[d] \\
 \{0\} \arrow[r] & X \cocartesian.
\end{tikzcd}
\]
The identification makes the following composition possible:
\[
\begin{tikzpicture}[x=1cm,y=1cm,>=stealth,baseline=(current bounding box.center)]
 \node (x) at (0,1.3) {$\alpha$};
 \node (a) at (2,1.3) {$a$};
 \node (b) at (4,1.3) {$b$};
 \node (y) at (6,1.3) {$\beta$};
 \node (qx) at (0,0) {$[\alpha]$};
 \node (q) at (3,0) {$[a]=[b]$};
 \node (qy) at (6,0) {$[\beta]$};
 \draw[->] (x) -- node[above] {$x$} (a);
 \draw[->] (b) -- node[above] {$y$} (y);
 \draw[->] (qx) -- node[below] {$x$} (q);
 \draw[->] (q) -- node[below] {$y$} (qy);
 \draw[->,dashed] (x) -- (qx);
 \draw[->,dashed] (a) -- (q);
 \draw[->,dashed] (b) -- (q);
 \draw[->,dashed] (y) -- (qy);
 \node at (-0.8,1.3) {$A$};
 \node at (-0.8,0) {$X$};
\end{tikzpicture}
\]
A \emph{reduced word} of length \(|w|=m\) is a finite sequence
\[
w=((\alpha_1,\beta_1),\ldots,(\alpha_m,\beta_m)),\qquad m\geq1,
\]
satisfying
\[
q(\beta_i)=q(\alpha_{i+1}),\qquad
\beta_i\neq\alpha_{i+1}\qquad(1\leq i<m).
\]
Associate to it
\begin{equation}\label{eq:word-space}
\Path_wA=\Path_{\alpha_1,\beta_1}A\otimes\cdots\otimes\Path_{\alpha_m,\beta_m}A.
\end{equation}

\begin{proposition}\label{prop:R-paths}
The pushout has state set \(X^0=A^0/{\sim}\) and natural isomorphisms
\begin{equation}\label{eq:R-normal-form}
\Path_{\bar\alpha,\bar\beta}X\cong
\coprod_{\substack{w\ \mathrm{reduced}\\
q(\alpha_1)=\bar\alpha,\ q(\beta_m)=\bar\beta}}
\Path_wA.
\end{equation}
\end{proposition}

\begin{proof}
For a finite word with \(q(\beta_i)=q(\alpha_{i+1})\), define \(\operatorname{red}(w)\) by replacing each maximal consecutive block with \(\beta_i=\alpha_{i+1}\) by the pair consisting of its first source and last target. Composition in \(A\) gives a map from the tensor product for \(w\) to \(\Path_{\operatorname{red}(w)}A\). Tensor products preserve coproducts, so these maps define composition on the right side of \eqref{eq:R-normal-form}. For three reduced words,
\[
\operatorname{red}(\operatorname{red}(wv)z)
=\operatorname{red}(wvz)
=\operatorname{red}(w\operatorname{red}(vz)).
\]
The corresponding maps agree by associativity in \(A\) and functoriality of \(\otimes\): disjoint block compositions commute, and compositions sharing a factor satisfy
\[
\mu_A(\mu_A\otimes\id)=\mu_A(\id\otimes\mu_A).
\]
Thus \eqref{eq:R-normal-form}, with state set \(A^0/{\sim}\), defines a \(\Pcat\)-flow, and the summands of length one define \(A\to X\).

For a map \(h:A\to Y\) with \(h(a)=h(b)\), the extension on a summand of length \(m\) is
\[
\Path_wA\xrightarrow{\bigotimes_{i=1}^m\Path_{\alpha_i,\beta_i}h}
\bigotimes_{i=1}^m\Path_{h(\alpha_i),h(\beta_i)}Y
\xrightarrow{\mu_Y^{(m)}}
\Path_{h(\alpha_1),h(\beta_m)}Y,
\]
where \(\mu_Y^{(m)}\) is iterated composition and \(\mu_Y^{(1)}=\id\). Every word is a composite of its factors, so the extension is unique. Consequently,
\[
\PFlow(X,Y)\cong\{h\in\PFlow(A,Y)\mid h(a)=h(b)\}.\qedhere
\]
\end{proof}

\begin{proposition}[Identifying two states]\label{lem:R-step}
Let \(a,b\in A^0\), and let \(X\) be the pushout of \(A\) along \(R\) at these states. Then \(X^0=A^0/{\sim}\), and \(\Path A(L)\to\Path X(L)\) is an h-cofibration for every \(L\in\Obj(\Pcat)\). If \(s:A\to A'\) is a weak equivalence and \(X'\) is the corresponding pushout, then \(X\to X'\) is a weak equivalence.
\end{proposition}

\begin{proof}
If \(a=b\), then \(X=A\). Otherwise apply Proposition~\ref{prop:R-paths}. The summands of length one give
\[
\Path X\cong\Path A\amalg W,
\qquad W=\coprod_{\substack{w\ \mathrm{reduced}\\|w|\geq2}}\Path_wA.
\]
Hence the following square is a pushout:
\[
\begin{tikzcd}
 \varnothing \arrow[r] \arrow[d] & W(L) \arrow[d] \\
 \Path A(L) \arrow[r] & \Path X(L) \cocartesian.
\end{tikzcd}
\]
Its bottom map is an h-cofibration because every space is h-cofibrant. The bijection \(s^0\) identifies the reduced words for \(A,A'\) and their quotient state sets. Under \eqref{eq:R-normal-form}, the induced map on execution paths is
\[
\coprod_w\bigotimes_{i=1}^{|w|}\Path_{\alpha_i,\beta_i}s.
\]
It is an objectwise weak homotopy equivalence by Lemma~\ref{lem:finite-tensor-coproduct}.
\end{proof}

\begin{thm}\label{thm:left-proper}
For every reparametrization category with cuts \(\Pcat\), the q-model structure on \(\PFlow\) is left proper.
\end{thm}

\begin{proof}
First let \(i:A\to X\) be a relative cell complex generated by \eqref{eq:Pflow-generators}. Well-ordering its cells refines it to a transfinite tower of pushouts of single generating maps,
\[
A=A_0\longrightarrow A_1\longrightarrow\cdots\longrightarrow A_\lambda=X.
\]
For a weak equivalence \(s:A\to A'\), form the pushouts
\[
\begin{tikzcd}[column sep=large]
 A \arrow[r] \arrow[d,"s"']
   & A_\mu \arrow[r] \arrow[d,"s_\mu"'] & A_{\mu+1} \arrow[d,"s_{\mu+1}"] \\
 A' \arrow[r] & A'_\mu \cocartesian \arrow[r] & A'_{\mu+1} \cocartesian.
\end{tikzcd}
\qquad A'_\mu=A'\amalg_A A_\mu.
\]
We prove by transfinite induction that \(s_\mu\) is a weak equivalence. Simultaneously, every successor map
\begin{equation}\label{eq:successor-hcof}
\Path A_\mu(L)\longrightarrow\Path A_{\mu+1}(L),
\qquad
\Path A'_\mu(L)\longrightarrow\Path A'_{\mu+1}(L)
\end{equation}
is an h-cofibration for every \(L\in\Obj(\Pcat)\). At successor stages, apply Proposition~\ref{lem:R-step} for \(R\) and Proposition~\ref{prop:globe-step} for a globe. A pushout of \(C\) satisfies
\[
A_{\mu+1}^0=A_\mu^0\amalg\{*\},
\qquad\Path A_{\mu+1}=\Path A_\mu,
\]
with the same identities for \(A'\), so it also preserves both assertions.

At a limit stage \(\mu\), the path towers consist objectwise of relative-\(T_1\) inclusions \cite[Proposition~2.6]{leftproperflow}. Therefore \cite[Theorem~6.14]{Moore1} gives
\begin{equation}\label{eq:path-limit}
\liminj_{\nu<\mu}\Path A_\nu\xrightarrow{\cong}\Path A_\mu,
\qquad
\liminj_{\nu<\mu}\Path A'_\nu\xrightarrow{\cong}\Path A'_\mu.
\end{equation}
Corollary~\ref{cor:transfinite-gluing} makes the comparison an objectwise weak homotopy equivalence. On states,
\[
s_\mu^0=\liminj_{\nu<\mu}s_\nu^0
\]
is a bijection, being the colimit of a natural isomorphism. This proves the cellular case.

For an arbitrary q-cofibration, the small object argument \cite[Theorem~2.1.14 and Corollary~2.1.15]{MR99h:55031} gives
\[
i=pj,
\qquad A\xrightarrow{j}Y\xrightarrow{p}X,
\]
with \(j\) a relative cell complex and \(p\) a trivial q-fibration. The lifting problem
\[
\begin{tikzcd}[column sep=large,row sep=large]
 A \arrow[r,"j"] \arrow[d,"i"'] & Y \arrow[d,"p"] \\
 X \arrow[r,"\id_X"'] \arrow[ur,dashed,"r"] & X
\end{tikzcd}
\]
gives \(ri=j\) and \(pr=\id_X\). The pushout functor \(A'\amalg_A-\) preserves this retract. With \(X_s=A'\amalg_A X\) and \(Y_s=A'\amalg_A Y\), we obtain
\[
\begin{tikzcd}[column sep=large]
X \arrow[r,"r"] \arrow[d,"\bar s_i"']
& Y \arrow[r,"p"] \arrow[d,"\bar s_j"] & X \arrow[d,"\bar s_i"]\\
X_s \arrow[r,"r_s"'] & Y_s \cocartesian \arrow[r,"p_s"'] & X_s \cocartesian,
\end{tikzcd}
\qquad pr=\id_X,\quad p_sr_s=\id_{X_s}.
\]
Since \(\bar s_j\) is a weak equivalence and weak equivalences are closed under retracts, \(\bar s_i\) is a weak equivalence.
\end{proof}

\section{Applications to reparametrization categories}
\label{sec:7}

The objects of \(\G\) and \(\M\) are the positive real numbers, with tensor product \(a\otimes b=a+b\).  Their enriched hom-spaces are
\begin{align*}
 \G(a,b)&=\{\text{nondecreasing homeomorphisms }[0,a]\to[0,b]\},\\
 \M(a,b)&=\{\text{nondecreasing continuous surjections }[0,a]\to[0,b]\}.
\end{align*}
They carry the \(\Delta\)-kelleyfication of the relative compact-open topology, and composition is composition of maps. These data define reparametrization categories by \cite[Propositions~4.9 and~4.11]{Moore1}. For \(\alpha\in\Pcat(a,a')\) and \(\beta\in\Pcat(b,b')\), where \(\Pcat\in\{\G,\M\}\), the tensor product of morphisms is
\begin{equation}\label{eq:block}
 (\alpha\otimes\beta)(t)=
 \begin{cases}
  \alpha(t),&0\leq t\leq a,\\
  a'+\beta(t-a),&a\leq t\leq a+b.
 \end{cases}
\end{equation}
This is \cite[Definition~4.6]{Moore1}; its associativity and continuity are part of \cite[Propositions~4.7, 4.9 and~4.11]{Moore1}.

\begin{lem} \label{lem:inverse}
The function
	\[
	(-)^{-1}:\G(L,L') \to \G(L',L),\qquad f\longmapsto f^{-1}
	\]
is continuous.
\end{lem}

\begin{proof}
This is \cite[Lemma~6.2]{Moore2}.
\end{proof}

\begin{lem}\label{lem:interval-adjoints}
Let \(\Pcat\in\{\G,\M\}\), let \(L,T>0\), and let \(Z\) be a space of \(\Top\). A function \(F:Z\to\Pcat(L,T)\) is continuous if and only if its adjoint
\[
\widetilde F:Z\times[0,L]\longrightarrow[0,T],\qquad
\widetilde F(z,t)=F(z)(t),
\]
is continuous.
\end{lem}

\begin{proof}
The enriched hom-space \(\Pcat(L,T)\) has the \(\Delta\)-kelleyfication of the relative compact-open topology. Cartesian closedness identifies a continuous adjoint with a continuous map into \(\TOP([0,L],[0,T])\). Since all its values belong to \(\Pcat(L,T)\), the universal property of the relative topology and of \(\Delta\)-kelleyfication gives the required factorization; see \cite[Section~2]{Moore1}. The converse follows from continuity of evaluation.
\end{proof}

\begin{proposition}[Cut operators for intervals]\label{prop:interval-cut-operators}
For \(\Pcat\in\{\G,\M\}\) and \(a,b,L>0\), the following formulas define continuous cut operators, natural in \(a\) and \(b\). For \(\psi\in\Pcat(L,a+b)\), if \(\Pcat=\G\), let \(\tau=\psi^{-1}(a)\) and put
\begin{align}
 c^{a,b}_1(\psi)(t)
  &=\psi\left(\frac{\tau t}{L}\right),                \label{eq:G-cut1}\\
 c^{a,b}_2(\psi)(t)
  &=\psi\left(\tau+\frac{(L-\tau)t}{L}\right)-a.     \label{eq:G-cut2}
\end{align}
If \(\Pcat=\M\), put
\begin{align}
 c^{a,b}_1(\psi)(t)&=\min\{\psi(t),a\},                \label{eq:M-cut1}\\
 c^{a,b}_2(\psi)(t)&=\max\{\psi(t)-a,0\}.             \label{eq:M-cut2}
\end{align}
\end{proposition}

\begin{proof}
For \(\G\), the cuts are strictly increasing with endpoint values \(0,a\) and \(0,b\). For \(\M\), they are continuous and nondecreasing with the same endpoint values, hence surjective. Inversion is continuous by Lemma~\ref{lem:inverse}, and evaluation is continuous by cartesian closedness, so
\[
\tau:\G(L,a+b)\longrightarrow]0,L[,
\qquad\psi\longmapsto\psi^{-1}(a)
\]
is continuous. The adjoints of \eqref{eq:G-cut1}--\eqref{eq:M-cut2} are compositions of evaluation with continuous real functions; Lemma~\ref{lem:interval-adjoints} proves continuity of the cuts.

\begin{figure}[htbp]
\centering
\begin{tikzpicture}[x=1cm,y=1cm,>=stealth,font=\small]
 \begin{scope}
  \draw[->] (0,0) -- (4.5,0) node[right] {$t$};
  \draw[->] (0,0) -- (0,3.55);
  \node[below left] at (0,0) {$0$};
  \node[below] at (4,0) {$L$};
  \node[left] at (0,3) {$a+b$};
  \node[left] at (0,1.3) {$a$};
  \draw[densely dotted] (0,1.3) -- (2,1.3) -- (2,0);
  \draw[densely dotted] (0,3) -- (4,3) -- (4,0);
  \node[below] at (2,0) {$\tau$};
  \draw[very thick] (0,0) -- (1,.5) -- (2,1.3) -- (4,3);
  \node[above left] at (3.5,2.6) {$\psi$};
  \fill (2,1.3) circle (1.7pt);
  \node at (2,-.8) {$\mathcal G:\quad \psi(\tau)=a$};
 \end{scope}
 \begin{scope}[xshift=7cm]
  \draw[->] (0,0) -- (4.5,0) node[right] {$t$};
  \draw[->] (0,0) -- (0,3.55);
  \node[below left] at (0,0) {$0$};
  \node[below] at (4,0) {$L$};
  \node[left] at (0,3) {$a+b$};
  \node[left] at (0,1.2) {$a$};
  \node[left] at (0,1.8) {$b$};
  \draw[densely dotted] (0,3) -- (4,3) -- (4,0);
  \draw[densely dotted] (0,1.8) -- (4,1.8);
  \draw[thick,dashed] (0,0) -- (1,.7) -- (1.6,1.2) -- (2.4,1.2) -- (4,3);
  \node[above left] at (3.6,2.55) {$\psi$};
  \draw[very thick] (0,0) -- (1,.7) -- (1.6,1.2) -- (4,1.2);
  \node[right] at (4,1.2) {$c_1(\psi)$};
  \draw[very thick,dash dot] (0,0) -- (2.4,0) -- (4,1.8);
  \node[right] at (4,1.8) {$c_2(\psi)$};
  \node at (2,-.8) {$\mathcal M:\quad c_1(\psi)=\min\{\psi,a\},\quad c_2(\psi)=\max\{\psi-a,0\}$};
 \end{scope}
\end{tikzpicture}
\caption{The cuts for \(\G\) and \(\M\). In \(\G\), the restrictions of \(\psi\) to \([0,\tau]\) and \([\tau,L]\) are reparametrized on \([0,L]\), and the second is translated by \(-a\). In \(\M\), the cut operators allow constant intervals.}
\label{fig:interval-cuts}
\end{figure}

For \(\G\), put \(\chi=(\alpha\otimes\beta)\psi\). By \eqref{eq:block},
\[
\chi^{-1}(a')=\psi^{-1}(a)=\tau.
\]
The two parts of \eqref{eq:block} give
\begin{align*}
c^{a',b'}_1(\chi)(t)
&=\alpha\!\left(\psi\!\left(\frac{\tau t}{L}\right)\right)
=\alpha(c^{a,b}_1(\psi)(t)),\\
c^{a',b'}_2(\chi)(t)
&=\beta\!\left(\psi\!\left(\tau+\frac{(L-\tau)t}{L}\right)-a\right)
=\beta(c^{a,b}_2(\psi)(t)).
\end{align*}
For \(\M\), define \(q_1^{a,b}(z)=\min\{z,a\}\) and \(q_2^{a,b}(z)=\max\{z-a,0\}\). Then \(c_i^{a,b}(\psi)=q_i^{a,b}\psi\), and \eqref{eq:block} gives
\[
\begin{aligned}
\bigl(q_1^{a',b'},q_2^{a',b'}\bigr)((\alpha\otimes\beta)(z))
&=\begin{cases}
(\alpha(z),0),&0\leq z\leq a,\\
(a',\beta(z-a)),&a\leq z\leq a+b
\end{cases}\\
&=\bigl(\alpha(q_1^{a,b}(z)),\beta(q_2^{a,b}(z))\bigr).
\end{aligned}
\]
Composing with \(\psi\) proves, for both categories,
\begin{equation}\label{eq:GM-equiv}
\begin{aligned}
c^{a',b'}_1((\alpha\otimes\beta)\psi)&=\alpha c^{a,b}_1(\psi),\\
c^{a',b'}_2((\alpha\otimes\beta)\psi)&=\beta c^{a,b}_2(\psi).
\end{aligned}\qedhere
\end{equation}
\end{proof}

\begin{proposition}[Homotopies of reparametrizations]\label{prop:interval-homotopy}
Let \(\Pcat\in\{\G,\M\}\) and \(L,T>0\). There is a continuous map
\[
\Phi_{L,T}:[0,1]\times\Pcat(L,T)\times\Pcat(L,T)\longrightarrow\Pcat(L,T)
\]
with \(\Phi_{L,T}(0,f,g)=f\), \(\Phi_{L,T}(1,f,g)=g\), and
\[
\Phi_{L,T'}(u,kf,kg)=k\Phi_{L,T}(u,f,g)
\qquad\bigl(k\in\Pcat(T,T')\bigr).
\]
\end{proposition}

\begin{proof}
Put
\[
H_v(f,g)(t)=\max\{f(t),g(vt)\}\qquad(0\leq v\leq1).
\]
This is a continuous nondecreasing map with endpoints \(0,T\), hence a surjection. If \(f,g\) are strictly increasing and \(v>0\), then
\[
s<t\Longrightarrow
\max\{f(s),g(vs)\}<\max\{f(t),g(vt)\}.
\]
For \(v=0\), it equals \(f\). Thus \(H_v(f,g)\in\Pcat(L,T)\) in either case. Its adjoint is continuous by continuity of evaluation and maximum, so Lemma~\ref{lem:interval-adjoints} gives continuity into the hom-space. For any nondecreasing \(k\),
\[
H_v(kf,kg)(t)
=\max\{k(f(t)),k(g(vt))\}
=k\bigl(\max\{f(t),g(vt)\}\bigr).
\]
The identities \(H_0(f,g)=f\) and \(H_1(f,g)=H_1(g,f)=\max\{f,g\}\) give
\[
\begin{tikzcd}[column sep=large]
f\arrow[r,rightsquigarrow,"{H(-,f,g)}"]
&\max\{f,g\}
&g\arrow[l,rightsquigarrow,"{H(-,g,f)}"'].
\end{tikzcd}
\]
Concatenation defines
\[
\Phi_{L,T}(u,f,g)=
\begin{cases}
H_{2u}(f,g),&0\leq u\leq\frac12,\\
H_{2-2u}(g,f),&\frac12\leq u\leq1.
\end{cases}
\]
The pasting lemma gives continuity; the identities above give the endpoints and compatibility with postcomposition.
\end{proof}

\begin{thm}\label{thm:GM-cuttable}
The reparametrization categories \(\G\) and \(\M\) have cuts.
\end{thm}

\begin{proof}
Use the cut operators of Proposition~\ref{prop:interval-cut-operators} and the linear homeomorphism
\begin{equation}\label{eq:GM-delta}
\delta_L:[0,L]\longrightarrow[0,2L],\qquad\delta_L(t)=2t.
\end{equation}
The continuous map
\begin{equation}\label{eq:GM-N}
N_L^{a,b}(\psi)=\bigl(c_1^{a,b}(\psi)\otimes c_2^{a,b}(\psi)\bigr)\delta_L
\end{equation}
is natural by \eqref{eq:GM-equiv}. For \(k=\alpha\otimes\beta\),
\begin{equation}\label{eq:GM-N-equiv}
N_L^{a',b'}(k\psi)
=\bigl(\alpha c_1^{a,b}(\psi)\otimes\beta c_2^{a,b}(\psi)\bigr)\delta_L
=kN_L^{a,b}(\psi).
\end{equation}
Proposition~\ref{prop:interval-homotopy} supplies
\begin{equation}\label{eq:GM-Theta}
\Theta_L^{a,b}(u,\psi)=\Phi_{L,a+b}(u,\psi,N_L^{a,b}(\psi)),
\end{equation}
with endpoints \(\psi,N_L^{a,b}(\psi)\). Its naturality is
\[
\begin{aligned}
\Theta_L^{a',b'}(u,k\psi)
&=\Phi_{L,a'+b'}(u,k\psi,kN_L^{a,b}(\psi))\\
&=k\Phi_{L,a+b}(u,\psi,N_L^{a,b}(\psi))
=k\Theta_L^{a,b}(u,\psi).
\end{aligned}
\]
These data satisfy Definition~\ref{def:cuttable}.
\end{proof}

\begin{cor}\label{cor:GM-left-proper}
The q-model structures of \(\G\)-flows and \(\M\)-flows are left proper.
\end{cor}

\begin{proof}
Combine Theorem~\ref{thm:GM-cuttable} with Theorem~\ref{thm:left-proper}.
\end{proof}

\section{The case of the terminal category}
\label{sec:8}

Let \(\One\) be the category with one object \(*\) and one morphism \(\id_*\).  Its enriched hom-space is the one-point space, and its unique strict semimonoidal structure is given by \(*\otimes*=*\) and \(\id_*\otimes\id_*=\id_*\).  It is terminal because every \(\Top\)-enriched category admits a unique enriched functor to \(\One\).

\begin{proposition}\label{prop:one-cuttable}
The terminal category \(\One\) is a reparametrization category with cuts.
\end{proposition}

\begin{proof}
Apply Lemma~\ref{lem:one-object-reparametrization} to the one-point monoid. The unique maps
\[
\delta_*=\id_*,\qquad c_1^{*,*}(\id_*)=c_2^{*,*}(\id_*)=\id_*,
\qquad\Theta_*^{*,*}(u,\id_*)=\id_*
\]
satisfy Definition~\ref{def:cuttable}.
\end{proof}

An enriched presheaf on \(\One\) is just a space.  Formula \eqref{eq:coend} has one summand and only identity relations, so it gives the natural homeomorphism
\[
 (D\otimes E)(*)\cong D(*)\times E(*).
\]
Consequently, a \(\One\)-flow is precisely a small topologically enriched semicategory, i.e., an ordinary flow.  This identification is also \cite[Definition~17]{Moore3}.

\begin{cor} \label{cor:one-left-proper}
The q-model structure of ordinary flows is left proper.
\end{cor}

\begin{proof}
Under the preceding identification, apply Proposition~\ref{prop:one-cuttable} and Theorem~\ref{thm:left-proper}.  This recovers \cite[Theorem~5.6]{leftproperflow}.
\end{proof}

\section{A symmetric example with a strict unit}
\label{sec:9}

A unit for a \(\Top\)-enriched semimonoidal category \(\mathcal C\) consists of an object \(e\) and enriched natural isomorphisms
\[
\lambda_X: e\otimes X\xrightarrow{\cong}X,
\qquad
\rho_X: X\otimes e\xrightarrow{\cong}X
\]
satisfying the monoidal unit coherence \cite[Definition~2.2.8]{Tensor_categories}. For a strictly associative tensor product, the triangle axiom reads \(\id_X\otimes\lambda_Y=\rho_X\otimes\id_Y\).

A strict unit would require \(e+L=L\) with \(e>0\). The following argument excludes arbitrary unit constraints as well.

\begin{lem}\label{lem:interval-fixed-points}
For \(\Pcat\in\{\G,\M\}\) and \(L>0\), the points of \([0,L]\) fixed by every automorphism of \(L\) are exactly \(0\) and \(L\).
\end{lem}

\begin{proof}
Every morphism preserves the endpoints. For \(x,y\in]0,L[\), let \(\alpha:[0,x]\to[0,y]\) and \(\beta:[0,L-x]\to[0,L-y]\) be the linear homeomorphisms. Their tensor product is an automorphism of \(L\) sending \(x\) to \(y\). Thus the automorphisms act transitively on \(]0,L[\), and no interior point is fixed by all of them.
\end{proof}

\begin{proposition}\label{prop:GM-nonunital}
Neither \(\G\) nor \(\M\) is unital.  More strongly, for \(\Pcat\in\{\G,\M\}\) there do not exist an object \(e\) and an ordinary natural isomorphism
	\[
	\lambda_L: e\otimes L\xrightarrow{\cong}L
	\qquad(L\in\Obj(\Pcat)).
	\]
\end{proposition}

\begin{proof}
Suppose that an object \(e>0\) and a natural isomorphism \(\lambda:e\otimes-\to\id\) exist. Fix \(L>0\). In either category, an isomorphism \(\lambda_L:[0,e+L]\to[0,L]\) is a strictly increasing homeomorphism: it is an injective nondecreasing continuous map with a continuous inverse. Hence
\begin{equation}\label{eq:unitor-interior-point}
 a=\lambda_L(e)\in]0,L[.
\end{equation}
For every automorphism \(\phi\) of \(L\), naturality gives the commutative square
\[
\begin{tikzcd}
 e+L \arrow[r,"\id_e\otimes\phi"] \arrow[d,"\lambda_L"']
 & e+L \arrow[d,"\lambda_L"] \\
 L \arrow[r,"\phi"'] & L.
\end{tikzcd}
\]
Formula \eqref{eq:block} gives \((\id_e\otimes\phi)(e)=e\). Evaluating the square at \(e\), we obtain
\[
 \phi(a)=\phi(\lambda_L(e))
 =\lambda_L((\id_e\otimes\phi)(e))
 =\lambda_L(e)=a.
\]
Thus \(a\) is an interior point fixed by every automorphism of \(L\), contrary to Lemma~\ref{lem:interval-fixed-points}. No such natural isomorphism exists, so neither category admits a monoidal unit.
\end{proof}

Let \(\Pmult\) have one object \(*\), enriched hom-space \(\Pmult(*,*)=[0,1]\), and identity \(1\). Composition and tensor product are multiplication:
\begin{equation}\label{eq:tensor-multiplication}
 t\circ s=ts=s\otimes t\qquad(s,t\in[0,1]).
\end{equation}

\begin{thm}\label{thm:example}
There is a strict symmetric monoidal structure and a cut structure on \(\Pmult\). With these structures, \(\Pmult\) is a symmetric reparametrization category with cuts. Moreover, \(\Pmult\) is not equivalent to \(\One\).
\end{thm}

\begin{proof}
The interval \([0,1]\), with multiplication and identity \(1\), is a contractible commutative topological monoid. Lemma~\ref{lem:one-object-reparametrization} therefore gives the reparametrization category, the strict unit, and the symmetry.

For the cut structure, put
\[
\delta_*=1,\qquad c^{*,*}_1(\psi)=0,\qquad
c^{*,*}_2(\psi)=0\qquad(\psi\in[0,1]).
\]
The cut operators are continuous. Multiplication by \(0\) gives naturality and the value of \(N\):
\begin{align*}
c^{*,*}_1((\alpha\otimes\beta)\psi)&=0=\alpha c^{*,*}_1(\psi),\\
c^{*,*}_2((\alpha\otimes\beta)\psi)&=0=\beta c^{*,*}_2(\psi),\\
N^{*,*}_*(\psi)&=(0\otimes0)\delta_*=0.
\end{align*}
Use the homotopy to \(0\) given by
\begin{equation}\label{eq:theta-example}
\Theta^{*,*}_*(u,\psi)=(1-u)\psi.
\end{equation}
It is continuous, and its endpoints are \(\psi\) and \(N^{*,*}_*(\psi)\). Its naturality follows from
\begin{align*}
\Theta^{*,*}_*(u,(\alpha\otimes\beta)\psi)
&=(1-u)\alpha\beta\psi\\
&=\alpha\beta(1-u)\psi
=(\alpha\otimes\beta)\Theta^{*,*}_*(u,\psi).
\end{align*}
These data satisfy Definition~\ref{def:cuttable}.

The unique functor \(\Pmult\to\One\) is not faithful since \(0\neq1\) in \(\Pmult(*,*)\). Hence the underlying categories, and therefore the enriched categories, are not equivalent.
\end{proof}


\end{document}